\documentclass{amsart}

\usepackage{amssymb}
\usepackage{amsmath,amscd}
\usepackage{graphicx}
\usepackage{color,hyperref}

\newcommand{\mf}{\mathfrak}
\newcommand{\mc}{\mathcal}
\newcommand{\mb}{\mathbf}

\newcommand{\R}{\mathbf R}
\newcommand{\C}{\mathbf C}
\newcommand{\Q}{\mathbf Q}
\newcommand{\Z}{\mathbf Z}
\newcommand{\F}{\mathbf F}

\newcommand{\Fix}{\textnormal{Fix}}

\newcommand{\sm}[4]{{
  \bigl[\begin{smallmatrix}{#1}&{#2}\\{#3}&{#4}\end{smallmatrix}\bigr]}}

\numberwithin{equation}{section}

\theoremstyle{plain}
\newtheorem{theorem}{Theorem}[section]

\theoremstyle{definition}

\begin{document}

\title[Value sets of bivariate folding polynomials]{Value sets of bivariate 
folding polynomials \\ over finite fields}
\author{ \"{O}mer K\"{u}\c{c}\"{u}ksakall{\i}}
\address{Middle East Technical University, Mathematics Department, 06800 Ankara,
Turkey.}
\email{komer@metu.edu.tr}

\date{\today}

\begin{abstract}
We find the cardinality of the value sets of polynomial maps associated with 
simple complex Lie algebras $B_2$ and $G_2$ over finite fields. We achieve this 
by using a characterization of their fixed points in terms of sums of roots of 
unity. 
\end{abstract}



\maketitle

\section*{Introduction}
Let $q$ be a power of a prime $p$. Given a polynomial $f\in\Z[\mb{x}]$ with $n$ 
variables, we write $\bar{f}$ for the induced map over $\F_q$. If $\bar{f}: 
\F_q^n \rightarrow \F_q^n$ is not a bijection, then one may ask how far it is 
away from being a bijection. An approach to investigate this problem is to find 
the cardinality of the value set  $\bar{f}(\F_q^n)$. For an arbitrary polynomial 
map $f\in\Z[\mb{x}]$, there is no easy formula for this quantity. However, there 
are certain families with nice underlying algebraic structures which allow us 
to find the cardinality explicitly. An interesting single variable example is 
the family of Dickson polynomials for which a formula was found by Chou, 
Gomez-Calderon and Mullen 
\cite{chou}.

There is a generalization of Dickson polynomials, or Chebyshev polynomials, to 
several variables introduced by Lidl and Wells. They provide easy to check 
conditions for these functions to induce permutations over finite fields 
\cite{lidlwells}. Lidl and Wells achieve this by using the theory of symmetric 
polynomials together with some basic methods in the theory of finite fields. On 
the other hand, their construction can be related to the simple complex Lie 
algebras $A_n$ \cite{hoffwith}. In general, for an arbitrary Lie algebra 
$\mf{g}$, there is an associated infinite sequence of integrable polynomial 
mappings $P_\mf{g}^k$ determined from the conditions 
\[ \varPhi_\mf{g}(k\mathbf{x})=P_\mf{g}^k(\varPhi_\mf{g}(\mathbf{x})). \]
Here, the components of the vector function $\varPhi_\mf{g}$ are given by  
exponential sums which are obtained by the orbits of the Weyl group of $\mf{g}$.
All coefficients of the polynomials defining $P_\mf{g}^k$ are integers. This  
result was first given by Veselov \cite{veselov}, and somewhat later by 
Hofmann and Withers \cite{hoffwith}, independently. These maps $P_\mf{g}^k$ 
are also referred as folding polynomials \cite{withers}. This is because the 
parameter $k$ acts by folding over the underlying triangular fundamental region 
in the case of a rank two simple complex Lie algebra.

In our previous work \cite{bivariate}, we have provided easy to check conditions 
for the bivariate folding polynomials associated with $B_2$ and $G_2$ to 
induce permutations over finite fields. In this paper, we extend our results, by 
finding the cardinality of the value set for each member in those families, not 
only for the members that give permutations.

The organization of the paper is as follows: In the first section we give three 
examples which illustrate the idea that will be used for the further cases; the 
first example is the power maps which is the most elementary, the other two 
examples are the folding polynomials associated with the Lie algebras $A_1$ and 
$A_2$. In the second and the third sections, we consider the folding polynomials 
associated with $B_2$ and $G_2$, respectively. For each one of these two 
families, we prove a  formula for the cardinality of its value set over finite 
fields.

\section{Motivation}
In this section, we consider the cardinality of the value sets of three basic 
families, namely the power maps and folding polynomials associated with $A_1$ 
and $A_2$. We give alternative proofs of formulas for the cardinality of their 
value sets which will be a motivation for the further cases $B_2$ and $G_2$.

\subsection{The power maps} 
The nonzero elements of $\F_q$ can be parametrized by roots of 
unity. This parametrization is useful while studying the action of 
power maps $P_k(x)=x^k$ on such fields. Let $\mu_k \subset \C$ be the group of 
$k$-th roots of unity. Define $S(k) = \mu_{k}\cup\{0\}$. The set $S(k)$ can be 
described as the set of complex numbers satisfying the equation $x^{k+1}=x$. 
Given a function $f:\C^n \rightarrow \C^n$, we will use the notation 
$\Fix(f)=\{\mb{z}\in\C \mathrel{:} f(\mb{z})=\mb{z}\}$. In this case, we have 
$S(k) = \Fix(P_{k+1})$.

Let $\Q(\Fix(P_q))$ be the number field obtained by adjoining the solutions of 
the equation $x^q=x$, or elements of $\Fix(P_q)$, to rational numbers. Let 
$\mf{p}$ be a prime ideal of $\Q(\Fix(P_q))$ lying over $p$. The elements of 
$\F_q$ can be characterized as solutions of the equation $x^q=x$. Now 
$S(q-1)=\Fix(P_{q})$. Thus, there is a one-to-one correspondence 
\[ \F_q \longleftrightarrow S(q-1) \]
obtained by reducing the elements on the right hand side modulo $\mf{p}$.

The action of $\bar{P}_k(x)$ on the finite field $\F_q$ is compatible with 
the action of $P_k(x)$ on $S(q-1)$. From this point of view, it is now obvious 
that $\bar{P}_k$ induces a permutation of $\F_q$ if and only if $\gcd(k,q-1)=1$. 
Moreover, one can easily find the size of the value set. Set 
$a=(q-1)/\gcd(k,q-1)$. Then, 
\[ \bar{P}_k(\F_q)  = \bar{P}_k\left( \overline{S(q-1)} \right) = \overline{ 
P_k(S(q-1))} = \overline{ S(a) }.\]
The set $S(a) \subset \C$ has $a+1$ elements, namely $a$-th roots of unity 
together with the zero. Their reduction modulo $\mf{p}$ belongs to the finite 
field $\F_q$ and they are distinct. Thus, we have $|\bar{P}_k(\F_q)| = a + 1$.

\subsection{The case A1}
The $k$-th Dickson polynomial $D_k(x)\in\Z[x]$ is uniquely defined by the 
functional equation $D_k(y+y^{-1})=y^k+y^{-k}$. As an alternative approach, one 
can use the Lie algebra $A_1$ in order to produce same family of polynomials. 
Consider
\[\varPhi(\sigma)=\exp(2\pi i \sigma) + \exp(-2\pi i \sigma).\] 
Alternatively, the Dickson polynomial $D_k$ is the unique polynomial satisfying 
the equation $D_k(\varPhi(\sigma))=\varPhi(k\sigma)$. 

Let $K = \Q(\Fix(D_q))$, a number field obtained by adjoining the roots 
of the equation $D_q(x)=x$ to rational numbers. Let $\mf{p}$ be a prime 
ideal of $K$ lying over $p$. It is well known that $D_q(x) \equiv x^q 
\pmod{p}$. Thus, 
there is a one-to-one correspondence
\[ \F_q \longleftrightarrow \Fix(D_q) \]
obtained by reducing the algebraic elements in $\Fix(D_q)$ modulo $\mf{p}$. The 
elements in $\Fix(D_q)$ can be parametrized by $\varPhi$. Note that 
$\varPhi(\sigma)=\varPhi(\tilde{\sigma})$ if and only if $\sigma\equiv \pm 
\tilde{\sigma}\pmod{\Z}$. Moreover, we have $\Fix(D_q) = S_1 \cup S_2$ where
\[ S_1 = \left\{ \varPhi \left( \frac{s}{q-1} \right) \mathrel: s\in\Z  
\right\}\quad \textnormal{ and }\quad S_2 = \left\{ \varPhi \left( 
\frac{s}{q+1} \right) \mathrel: s\in\Z  \right\}.\]

Our purpose is to understand the size of $\bar{D}_k(\F_q)$. It is enough to 
investigate $D_k(S_1 \cup S_2)$ because of the one-to-one correspondence. 
Note that 
\[ \bar{D}_k(\F_q)  = \bar{D}_k\left(\, \overline{S_1 \cup S_2}\, \right) = 
\overline{ D_k(S_1 \cup S_2)}.\]
It is easy to see that
\[ D_k(S_1 \cup S_2) = \left\{ \varPhi \left( \frac{ks}{q-1} \right) 
\mathrel: s\in\Z  \right\} \cup \left\{ \varPhi \left( \frac{ks}{q+1} \right) 
\mathrel: s\in\Z  \right\}. \]
In order to find the precise number of elements in this union, one needs to be 
careful with possible common elements $\varPhi(0)$ and $\varPhi(1/2)$. The 
following result is a special case of a formula which was first established by 
Chou, Gomez-Calderon and Mullen \cite{chou}. A corollary of this theorem is the 
well known criterion, $\gcd(k,q^2-1)=1$, for the Dickson polynomials 
$\bar{D}_k$ being a permutation of $\F_q$.

\begin{theorem} 
Let $k$ be a positive integer. Set
\[ a=\frac{q-1}{\gcd(q-1,k)} \quad \textnormal{and} \quad 
b=\frac{q+1}{\gcd(q+1,k)}. \]
Then the cardinality of the value set is
\[ |\bar{D}_k(\F_q)| = \frac{a}{2} + \frac{b}{2} + \eta(k,q) \] 
where 
\[\eta(k,q) =\left\{\begin{array}{cl}
0 & \textnormal{ if } \gcd(a,2)=\gcd(b,2),\\
1/2 & \textnormal { if } \gcd(a,2) \neq \gcd(b,2).
\end{array}\right.\]
\end{theorem}
\begin{proof}
Our strategy is to separate this counting problem into two parts. Note that it 
is enough to consider $\varPhi(\sigma)$ with $0\leq \sigma < 1$ to represent 
any element in $\Fix(D_q)$. There are two elements, namely $\varPhi(0)$ and 
$\varPhi(1/2)$, whose representations are unique. Each other element is 
represented with precisely two different expressions, namely $\varPhi(\sigma)$ 
and $\varPhi(1-\sigma)$. Note that $\varPhi(\R)=[-2,2]$ and the elements 
$\varPhi(0)$ and $\varPhi(1/2)$ are the endpoints of this interval. We use this 
geometric interpretation to separate these two distinct types of elements as 
interior points and end points.

The following table gives the number of elements of each type in $D_k(S_1)$ and 
$D_k(S_2)$, respectively.
\[\begin{array}{|c|c|c|} \hline
× & \text{Interior} & \text{End} \\ \hline
D_k(S_1) & (a-\gcd(a,2))/2 & \gcd(a,2) \\ \hline
D_k(S_2) & (b-\gcd(b,2))/2 & \gcd(b,2) \\ \hline
\end{array}\]

Note that the set of interior elements of $D_k(S_1)$ and $D_k(S_2)$ are 
disjoint since $\gcd(a,b)$, which is a divisor of $\gcd(q-1,q+1)$, is either 
one or two. However this is not the case for the end points. The point 
$\varPhi(0)$, possibly $\varPhi(1/2)$, belongs to each one of these two sets. 

Now we are ready to establish the formula in the theorem. Suppose that 
$\gcd(a,2)=\gcd(b,2)$. Then, we have
\[ |\bar{D}_k(\F_q)| = \frac{a-\gcd(a,2)}{2} + \frac{b-\gcd(a,2)}{2} + (a,2) 
= \frac{a}{2} + \frac{b}{2}.\]

Suppose that $\gcd(a,2)\neq\gcd(b,2)$. Then $\{\gcd(a,2),\gcd(b,2)\}=\{1,2\}$. 
In this case, we have
\[ |\bar{D}_k(\F_q)| = \frac{a}{2} + \frac{b}{2} - \frac{1}{2} - \frac{2}{2} + 
2 = \frac{a}{2} + \frac{b}{2} + \frac{1}{2}.\]
\end{proof}

\subsection{The case A2}
The main result of this part, namely Theorem~\ref{A2main}, was first proved 
in \cite{bicheby}. For the convenience of the reader, we will summarize the 
main notions adapted to the terminology of Lie algebras. Then we will give 
an elaborated proof of Theorem~\ref{A2main} which will be a motivation for the 
further cases $B_2$ and $G_2$. 

Let $\{\alpha_1, \alpha_2\}$ be a choice of simple roots for the Lie algebra 
$A_2$ with Cartan matrix 
$\sm{2}{-1}{-1}{2}.$
The transpose of this matrix, which is itself, transforms the fundamental 
weights, say $\omega_1$ and $\omega_2$, into the fundamental roots. We have
$\alpha_1=2\omega_1-\omega_2$ and $\alpha_2=-\omega_1+2\omega_2$. The orbit of 
$\omega_1$, under the action of the Weyl group, is $\{ \omega_1, 
\omega_2-\omega_1,-\omega_2 \}$. Similarly, the orbit of $\omega_2$ is $\{ 
\omega_2, \omega_1-\omega_2,-\omega_1 \}$. Set $\sigma:=\omega_2$ and 
$\tau=\omega_1-\omega_2$. With this new choice, the orbits appear simpler. More 
precisely, we have 
$\{\sigma, \tau, -\sigma-\tau\}$ and $\{-\sigma, -\tau, \sigma+ \tau\}$. One can 
consider $\varPhi=(\varphi_1,\varphi_2)$ with
\begin{align*}
 \varphi_1&=e^{2\pi i \sigma} + e^{2\pi i \tau} + e^{-2\pi i (\sigma+\tau)},\\
 \varphi_2 &= e^{-2\pi i\sigma}+e^{-2\pi i\tau}+e^{2\pi i (\sigma+\tau)}.
\end{align*}

Observe that $\varPhi(\sigma,\tau)$ is equal to any one of the following six 
expressions below which are given by the elements of the Weyl group: 
\[\begin{array}{cc|cc|cc}
\text{I} &\varPhi(\sigma,\tau) & \text{II} &\varPhi(\sigma,-\sigma-\tau) & 
\text{III} &\varPhi(\tau,-\sigma-\tau) \\
\text{IV} &\varPhi(\tau,\sigma) &\text{V} &\varPhi(-\sigma-\tau,\sigma) & 
\text{VI} &\varPhi(-\sigma-\tau,\tau).\\
\end{array}\]
Under these symmetries, the region $0 \leq \sigma, \tau < 1$ is separated into 
six parts, possibly having two components, which are mutually congruent to each 
other under the action of the Weyl group. We choose one of them as $R_{A_2}$. 
See Figure~\ref{fig:RA2}.
\begin{figure}[htbp]
    \centering
 \includegraphics[scale=0.8]{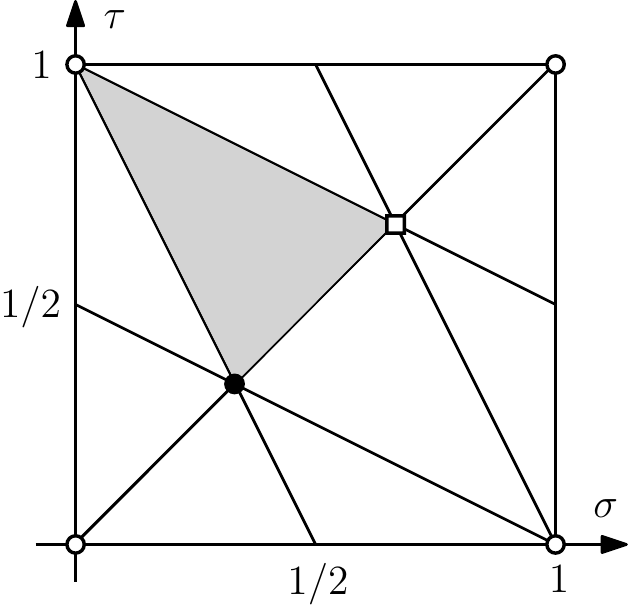}
    \caption{The fundamental region $R_{A_2}$.}
    \label{fig:RA2}
\end{figure}

The family of folding polynomials $\mc{A}_k$ satisfy the conditions 
\[\mc{A}_k(\varPhi(\sigma,\tau)) = \varPhi(k\sigma,k\tau).\]
We want to understand $\Fix(\mc{A}_q)$ in terms of $\varPhi(\cdot,\cdot)$. A 
fixed point of $\mc{A}_q$ is of the form $\varPhi(\sigma,\tau)$ where 
$(q\sigma,q\tau) \equiv w(\sigma,\tau) \pmod{\Z^2}$ for some $w$ in the Weyl 
group of the Lie algebra $A_2$. Using this setup, it is easy to show that 
$\Fix(\mc{A}_q)= \bigcup S_i$ where
\begin{align*}
 S_1 &= \left\{ \varPhi \left( \frac{s}{q-1},\frac{t}{q-1}\right) \mathrel: 
s,t\in\Z  \right\}, \\
 S_2 &= \left\{ \varPhi \left( \frac{s}{q^2-1},\frac{sq}{q^2-1} \right) 
\mathrel: s\in\Z  \right\}, \\
 S_3 &= \left\{ \varPhi \left( \frac{s}{q^2+q+1},\frac{sq}{q^2+q+1} \right) 
\mathrel: s\in\Z  \right\}.
\end{align*}

Let $K = \Q(\Fix(\mc{A}_q))$, a number field obtained by adjoining the 
solutions of $\mc{A}_q(x,y)=(x,y)$ to rational 
numbers. Let $\mf{p}$ be a prime ideal of $K$ lying over $p$. There is a 
one-to-one correspondence
\[ \F_q^2 \longleftrightarrow \Fix(\mc{A}_q) \]
obtained by reducing the elements on the right hand side modulo $\mf{p}$. 
Moreover this correspondence is compatible with the actions of $\bar{\mc{A}}_k$ 
and $\mc{A}_k$ on $\F_q^2$ and $\Fix(\mc{A}_q)$, respectively. Our purpose is to 
understand the size of
\[ \bar{\mc{A}}_k(\F_q^2)  = \bar{\mc{A}}_k\left(\, \overline{S_1 \cup S_2 \cup 
S_3}\, \right) = \overline{ \mc{A}_k(S_1) \cup \mc{A}_k(S_2) \cup 
\mc{A}_k(S_3)}.\]
In order to find the cardinality of the value set, it is enough to investigate 
the set of complex numbers $\bigcup \mc{A}_k(S_i)$.

\begin{theorem}\label{A2main}
Let $k$ be a positive integer. Set
\[a=\frac{q-1}{\gcd(k,q-1)},\ \ b=\frac{q^2-1}{\gcd(k,q^2-1)}\ \ \textnormal{ 
and }\ \ c=\frac{q^2+q+1}{\gcd(k,q^2+q+1)}.\] 
Then the cardinality of the value set is 
\[|\bar{\mc{A}}_k(\F_q^2)|=\frac{a^2}{6}+\frac{b}{2}+\frac{c}{3}+\eta(k,q)\]
where $\eta(k,q)$ is given by
\[\begin{array}{|c|c|c|} \hline
\eta(k,q) & 3\nmid k \textnormal{ or } 3\nmid a & 3\mathrel{|}k \textnormal{ and 
} 3\mathrel{|}a\\ 
\hline
2 \nmid k \textnormal{ or } 2\nmid b & 0 & 2/3\\ \hline
 2\mathrel{|}k \textnormal{ and } 2\mathrel{|} b & a/2 & a/2+2/3 \\ \hline
 \end{array}\]
In particular if $\gcd(k,6)=1$, then $\eta(k,q)=0$.
\end{theorem}
\begin{proof}
Our strategy is to separate the problem into three parts. We will consider 
points in the interior, on the edge and at the corners separately.

A point $\varPhi(\sigma,\tau)$, with $0 \leq \sigma, \tau < 1$, is a corner 
point if and only if it is of the form $\varPhi(0,0), \varPhi(1/3,1/3)$ or
$\varPhi(2/3,2/3)$.

A point $\varPhi(\sigma,\tau)$ is an edge point if it is given by a pair
$(\sigma,\tau)$ that is on the boundary of $R_{A_2}$ except the corners. An 
edge point can be expressed in one of the forms $\varPhi(\sigma,\sigma)$ 
$\varPhi(\sigma, -2\sigma)$ or $\varPhi(-2\sigma, \sigma)$. 

A point $\varPhi(\sigma,\tau)$ is an interior point if it is given by a pair
$(\sigma,\tau)$ that is in the interior of $R_{A_2}$. There are exactly six 
distinct representations given by I, II, III, IV, V and VI, when the components 
are considered modulo integers.  

We have the following table for the number of special types of points in each 
one of the sets $\mc{A}_k(S_i)$.
\[\begin{array}{|l|c|c|c|} \hline
× & \text{Interior} & \text{Edge} & \text{Corner}\\ \hline
\mc{A}_k(S_1) & (a^2-3a+2(a,3))/6 & a-\gcd(a,3)&\gcd(a,3)\\ \hline
\mc{A}_k(S_2) & (b-\gcd(q-1,b))/2  & \gcd(q-1,b)-\gcd(a,3) & \gcd(a,3) \\ \hline
\mc{A}_k(S_3) & (c-\gcd(c,3))/3  & 0 & \gcd(c,3) \\ \hline
\end{array}\]

We start with explaining the entries in the first column. The elements in 
$\mc{A}_k(S_1)$ are of the form $\varPhi(s/a,t/a)$ for some integers $s$ and 
$t$. There are $a^2$ pairs $(s/a,t/a)$ with $0\leq s,t < a$ and as a result 
there are roughly $a^2/6$ interior points in $\mc{A}_k(S_1)$. In order to find 
the precise number, we need to exclude pairs giving edge and corner points. 
For each $0\leq s <a$, the pairs $(s/a,s/a), (s/a,-2s/a)$ and $(-2s/a,s/a)$ 
give rise to an edge or a corner point. Observe that, if $3|a$, then the 
choices $s=a/3$ and $s=2a/3$ give rise to corner points $\varPhi(1/3,1/3)$ and 
$\varPhi(2/3,2/3)$, respectively. Thus, we find the following number:
\[ a^2 - 3(a-\gcd(a,3)) - \gcd(a,3) = a^2-3a+2\gcd(a,3).\]
Note that this number is divisible by six for each choice of $a$. This 
justifies the top entry in the first column. 

Secondly, we consider the interior points in $\mc{A}_k(S_2)$. The elements in 
$\mc{A}_k(S_2)$ are of the form $\varPhi(s/b,sq/b)$ for some integer $s$. There 
are $b$ such pairs with $0\leq s < b$. Unlike the previous case, an 
interior point of this form has only two representations, namely 
$\varPhi(s/b,sq/b)$ and $\varPhi(sq/b,s/b)$ where $sq$ is considered modulo 
$b$. This is because the multiplicative order of $q$ modulo $b$ is a divisor of 
$2$. Note that $s \equiv sq \pmod{b}$ if and only if $s$ is a multiple of 
$b/\gcd(q-1,b)$. The number of such multiples is $\gcd(q-1,b)$, each one of 
which gives an edge or a corner point. Thus, the number of pairs giving an 
interior point is equal to
\[ b-\gcd(q-1,b). \]
Note that this number is divisible by two for each choice of $b$. This 
justifies the middle entry in the first column. 

Next, we consider the interior points in $\mc{A}_k(S_3)$. It is clear that 
the multiplicative order of $q$ modulo $c$ is a divisor of three. If the order 
is one, then this means that $c=1$ or $c=3$. In such a case, we have a corner 
point. Otherwise, a generic point has three distinct expressions, namely
\[ \varPhi\left(\frac{s}{c},\frac{sq}{c}\right), \quad 
\varPhi\left(\frac{sq}{c},\frac{sq^2}{c}\right)\quad \text{and}\quad
\varPhi\left(\frac{sq^2}{c},\frac{s}{c}\right).\]
This proves the bottom entry in the first column. 

As we finish the discussion for the interior points, we also note that the set 
of interior points of $\mc{A}_k(S_i)$ are pairwise disjoint. Firstly, 
it is clear that the intersection of $\mc{A}_k(S_3)$ with $\mc{A}_k(S_1)$ 
consists of corner points only. This is because of the fact that 
$\gcd(q-1,q^2+q+1)$ is a divisor of $3$. A similar argument holds for the 
intersection $\mc{A}_k(S_3)$ with $\mc{A}_k(S_2)$ because 
$\gcd(q^2-1,q^2+q+1)$ is a divisor of three, too. Finally, suppose that $\alpha 
\in \mc{A}_k(S_1) \cap \mc{A}_k(S_2)$. Then
\[\alpha = \varPhi\left(\frac{s}{q-1},\frac{sq}{q-1}\right) = 
\varPhi\left(\frac{s}{q-1},\frac{s}{q-1}\right)\]
for some integer $s$. Thus, $\alpha$ is not an interior point.

Now we consider the edge and corner points. We claim that $\mc{A}_k(S_2)$ 
contains all possible edge and corner points. For example, if we pick a point 
$\alpha\in\mc{A}_k(S_1)$, then it is of the form $\alpha=\varPhi(s/a,t/a)$. 
If $\alpha$ is an edge point, then we must have $s \equiv t\pmod{a}$. It 
follows that $t \equiv sq \pmod{a}$ since $a$ is a divisor of $q-1$. Therefore, 
$\alpha = \varPhi(s/a,sq/a)$ and as a result $\alpha$ is an element of 
$\mc{A}_k(S_2)$. The other parts of the claim can be verified easily, and 
therefore omitted.

We use the following implications in order establish the formulas for 
$\eta(k,q)$ in the theorem:
\begin{enumerate}
 \item $2 \nmid k$ or $2\nmid b\Rightarrow \gcd(q-1,b)=a$
 \item $2 \mathrel{|} k$ and $2 \mathrel{|} b\Rightarrow \gcd(q-1,b)=2a$
 \item $3 \nmid k$ or $3\nmid a\Rightarrow \gcd(c,3)=\gcd(a,3)$
 \item $3 \mathrel{|} k$ and $3 \mathrel{|} a\Rightarrow \gcd(c,3)=1$ and 
$\gcd(a,3)=3$
\end{enumerate}

In order to establish the entry $a/2+2/3$ for $\eta(k,q)$ in the 
statement of the theorem, we shall use the implications (2) and (4). If (2) and 
(4) hold, then the cardinality of the value set is 
\[ \left( \frac{a^2-3a+2\cdot3}{6}+  \frac{b-2a}{2} + \frac{c-1}{3} \right) 
+\left( 2a-3 \right) + \left(3\right). \]
This is the number of elements in the union $\bigcup \mc{A}_k(S_i)$ obtained as 
a sum of the number of interior, edge and corner elements, respectively. This 
finishes the proof of the fact that $\eta(k,q)=a/2+2/3$. The proof of the other 
cases are similar.
\end{proof}

\section{The case B2}
Unless otherwise stated or proved, the assertions of this section can be found 
in \cite{bivariate}. For the convenience of the reader, we will summarize the 
main notions. Then we will prove the 
main result of this section, see Theorem~\ref{B2main}. 

Let $\{\alpha_1, \alpha_2\}$ be a choice of simple roots for the Lie algebra 
$B_2$ with Cartan matrix $\sm{2}{-2}{-1}{2}$. 
The transpose of this matrix transforms the fundamental weights into the 
fundamental roots. We have $\alpha_1=2\omega_1-\omega_2$ and 
$\alpha_2=-2\omega_1+2\omega_2$. The orbit of $\omega_1$, under the action of 
the Weyl group, is $\{ \pm \omega_1, \pm(2\omega_2-\omega_1) \}$. 
Similarly, the orbit of $\omega_2$ is $\{ \pm \omega_2, \pm(\omega_1-\omega_2) 
\}$. We set $\sigma:=\omega_2$ and $\tau=\omega_1-\omega_2$. With this new 
choice, the orbits appear simpler. More precisely, we have $ \{ \pm \sigma, \pm 
\tau  \} $ and  $ \{ \pm (\sigma+\tau), \pm (\sigma- \tau) \}$. One can 
consider $\varPhi=(\varphi_1,\varphi_2)$ with
\begin{align*}
 \varphi_1&=e^{2\pi i \sigma} + e^{-2\pi i \sigma} + e^{2\pi i 
\tau} + e^{-2\pi i \tau}\\
 \varphi_2 &= e^{2\pi i (\sigma+\tau)} + e^{-2\pi i (\sigma+\tau)} + e^{2\pi i 
(\sigma-\tau)} 
+ e^{2\pi i (\tau-\sigma)}
\end{align*}

Observe that $\varPhi(\sigma,\tau)$ is equal to any one of the following eight 
expressions below which are given by the elements of the Weyl group: 
\[\begin{array}{cc|cc|cc|cc}
\text{I} &\varPhi(\sigma,\tau) & \text{II} &\varPhi(\sigma,-\tau) & \text{III} 
&\varPhi(-\sigma,\tau) & \text{IV} &\varPhi(-\sigma,-\tau) \\
\text{V} &\varPhi(\tau,\sigma) &\text{VI} &\varPhi(-\tau,\sigma) & \text{VII} 
&\varPhi(\tau,-\sigma) & \text{VIII} &\varPhi(-\tau,-\sigma)\\
\end{array}\]
Under these symmetries, the region $0 \leq \sigma, \tau < 1$ is separated into 
eight triangles which are mutually congruent to each other under the action of 
the Weyl group. We choose one of them as $R_{B_2}$. See Figure~\ref{fig:RB2}.

\begin{figure}[htbp]
    \centering
 \includegraphics[scale=0.8]{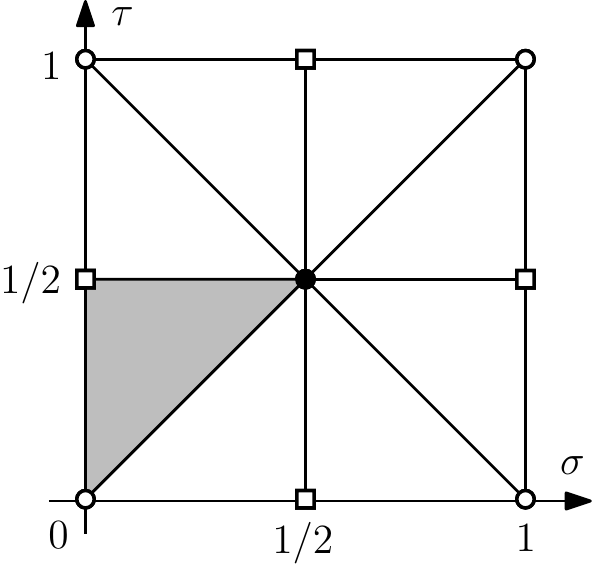}
    \caption{The fundamental region $R_{B_2}$.}
    \label{fig:RB2}
\end{figure}

The family of folding polynomials $\mc{B}_k$ satisfy the conditions 
\[\mc{B}_k(\varPhi(\sigma,\tau)) = \varPhi(k\sigma,k\tau).\]
We want to understand $\Fix(\mc{B}_q)$ in terms of $\varPhi(\cdot,\cdot)$. A 
fixed point of $\mc{B}_q$ is of the form $\varPhi(\sigma,\tau)$ where 
$(q\sigma,q\tau) \equiv w(\sigma,\tau) \pmod{\Z^2}$ for some $w$ in the Weyl 
group of $B_2$. Using this setup, it is not hard to show that $\Fix(\mc{A}_q)= 
\bigcup S_i$ where
\begin{align*}
 S_1 &= \left\{ \varPhi \left( \frac{s}{q-1},\frac{t}{q+1}\right) \mathrel: 
s,t\in\Z  \right\}, \\
 S_2 &= \left\{ \varPhi \left( \frac{s}{q-1},\frac{t}{q-1}\right) \mathrel: 
s,t\in\Z  \right\}, \\
 S_3 &= \left\{ \varPhi \left( \frac{s}{q+1},\frac{t}{q+1}\right) \mathrel: 
s,t\in\Z  \right\}, \\
 S_4 &= \left\{ \varPhi \left( \frac{s}{q^2-1},\frac{sq}{q^2-1} \right) 
\mathrel: s\in\Z  \right\}, \\
 S_5 &= \left\{ \varPhi \left( \frac{s}{q^2+1},\frac{sq}{q^2+1} \right) 
\mathrel: s\in\Z  \right\}.
\end{align*}

Let $K = \Q(\Fix(\mc{B}_q))$, a number field obtained by adjoining the solutions 
of $\mc{B}_q(x,y)=(x,y)$ to rational numbers. 
Let $\mf{p}$ be a prime ideal of $K$ lying over $p$. There is a one-to-one 
correspondence
\[ \F_q^2 \longleftrightarrow \Fix(\mc{B}_q) \]
obtained by reducing the elements on the right hand side modulo $\mf{p}$. 
Moreover this correspondence is compatible with the actions of $\bar{\mc{B}}_k$ 
and $\mc{B}_k$ on $\F_q^2$ and $\Fix(\mc{B}_q)$, respectively. Our purpose is 
to understand the size of
\[ \bar{\mc{B}}_k\left(\F_q^2\right)  = \bar{\mc{B}}_k\left( \overline{\bigcup 
S_i} 
\right) = \overline{ \bigcup \mc{B}_k(S_i) }.\]
Here $i$ runs from $1$ to $5$. In order to find the cardinality of the value 
set, it is enough to investigate the set of complex numbers $\bigcup 
\mc{B}_k(S_i)$.

\begin{theorem}\label{B2main}
Let $k\geq 1$ be an integer. Set
\[a=\frac{q-1}{(q-1,k)},\ \ 
b=\frac{q+1}{(q+1,k)},\ \ 
c=\frac{q^2-1}{(q^2-1,k)}\quad \textnormal{and}\quad
d=\frac{q^2+1}{(q^2+1,k)}.\] 
Then the cardinality of the value set is 
\[\left|\bar{\mc{B}}_k\left(\F_q^2\right)\right|=\frac{(a+b)^2}{8}+\frac{c+d}{4}
+\eta(k , q)\]
where
\[\eta(k,q)=
\left\{\begin{array}{lll} 
0 & \text{ if } & 2\nmid k \text{ or } 2\nmid c \\
(a+b)/2+1/8 & \text{ if } & 2\mathrel{|}k \text{ and } 2\mathrel{|}c \text{ and 
} 
2\mathrel{|}ab \\
(a+b)/4+1/4 & \text{ if } & 2\mathrel{|}k \text{ and } 2\mathrel{|}c \text{ and 
} 2\nmid ab \\
\end{array}\right.\]
\end{theorem}
\begin{proof}
Our strategy is to separate the problem into three parts according to 
the Figure~\ref{fig:RB2}. We will consider points in the interior, on the edge 
and at the corners separately.

A point $\varPhi(\sigma,\tau)$, with $0\leq \sigma, \tau <1$, is a 
corner point if and only if it is of the form $\varPhi(0,0), \varPhi(1/2,1/2), 
\varPhi(0,1/2)$ or $\varPhi(0,1/2)$.  Note that the last two points are the 
same.

A point $\varPhi(\sigma,\tau)$ is an edge point if it is given by a pair
$(\sigma,\tau)$ that is on the boundary of $R_{A_2}$ except the corners. An 
edge point can be expressed in the form $\varPhi(\sigma,\sigma), 
\varPhi(0,\sigma)$  or $\varPhi(1/2,\sigma)$. There are four distinct 
expressions for each edge point whose components restricted modulo integers.

A point $\varPhi(\sigma,\tau)$ is an interior point if it is given by a pair
$(\sigma,\tau)$ that is in the interior of $R_{B_2}$. There are eight 
different representations given by I, II, ..., VIII. 

To ease the notation, we set $m'=\gcd(m,2)$, for each integer $m$. We have the 
following table:
\[\begin{array}{|l|c|}\hline
× & \text{Interior} \\ \hline
\mc{B}_k(S_1) & (a-a')(b-b')/4 \\ \hline
\mc{B}_k(S_2) & (a-a')(a-a'-2)/8 \\ \hline
\mc{B}_k(S_3) & (b-b')(b-b'-2)/8 \\ \hline
\mc{B}_k(S_4) & (c-c/a-c/b+c')/4  \\ \hline
\mc{B}_k(S_5) & (d-d')/4 \\ \hline
\end{array}\]

The elements in $\mc{B}_k(S_1)$ are of the form $\varPhi(s/a,t/b)$ for 
some integers $s$ and $t$. There are $ab$ pairs $(s/a,t/a)$ with $0\leq s<a$ and 
$0\leq t<b$. Even though there are eight symmetries, it is not possible to 
switch the first and second component unless their denominator are both 
divisors of two. As a result there are roughly $ab/4$ interior points in 
$\mc{B}_k(S_1)$. In order to find the precise number, we need to exclude pairs 
giving edge and corner points. The number of suitable pairs is
\[ ab - a'b - b'a + a'b' = (a-a')(b-b')\]
which is obtained by applying the inclusion and exclusion principle. Note that 
this number is divisible by four for each choice of $a$ and $b$. This justifies 
the first entry in the table. 

Secondly, we consider the interior points in $\mc{B}_k(S_2)$. These points are 
of the form $\varPhi(s/a,t/a)$ for some integers $0\leq s,t < a$. Unlike the 
previous case, switching $\sigma$ and $\tau$ is allowed here. As a result, there 
are roughly $a^2/8$ interior points in $\mc{B}_k(S_2)$. We find the following 
number of pairs which give rise to an interior point:
\[ (a-a')(a-a')-2(a -a')= (a-a')(a-a'-2).\]
In this expression, the first term counts the pairs $(s/a,t/a)$ excluding the 
ones which are of the form $(\pm s/a,0)$ and possibly the ones which are of the 
form $(\pm s/a,1/2)$ if $a'=2$. The second term excludes the pairs of the form 
$(s/a,s/a)$. As we shall expect, this integer is always divisible by $8$. This 
justifies the entry for $\mc{B}_k(S_2)$. The computation is similar for 
$\mc{B}_k(S_3)$.

Next, we consider the elements in $\mc{B}_k(S_4)$. They are of the form 
$\varPhi(s/c,sq/c)$ for some integer $0\leq s<c$. The multiplicative order of 
$q$ modulo $c$ is a divisor of two. On the other hand we are allowed to switch 
$s$ with $c-s$. Thus, there are roughly $c/4$ interior points in this set. Note 
that $sq \equiv s \pmod{c}$ if and only if $s$ is a multiple of $a$ or $b$. 
Thus, the number of pairs giving an interior point is 
\[ c-c/a-c/b+\gcd(c/a,c/b). \]
Note that $\gcd(c/a,c/b)=c'$.

The points in $\mc{B}_k(S_5)$ are of the form $\varPhi(s/d,qs/d)$ for some $0 
\leq s < d$. Clearly, the multiplicative order of $q$ modulo $d$ is a divisor 
of four. Thus we obtain four different expressions:
\[ \varPhi\left(\frac{s}{d},\frac{sq}{d}\right),  
\varPhi\left(\frac{sq}{d},\frac{-s}{d}\right), 
\varPhi\left(\frac{-s}{d},\frac{-sq}{d}\right), 
\varPhi\left(\frac{-sq}{d},\frac{s}{d}\right)\]
Note that $q^2\equiv -1 \pmod{d}$. Each one of the expressions are distinct 
unless $s/d$ is congruent $1/2$ modulo integers. This justifies the last entry 
in the table.

Now, we focus on the edge points. We first note that $\mc{B}_k(S_5)$ has no 
such point. We claim that the following holds unless $2|k$ and $2 | ab$:
\[ \text{Edge}( \mc{B}_k(S_2) \cup \mc{B}_k(S_3) ) \subseteq  \text{Edge}( 
\mc{B}_k(S_1) \cup \mc{B}_k(S_4) ). \]
To see this pick an edge point $\alpha=\varPhi(s/a,t/a)$ from $\mc{B}_k(S_2)$ 
with $s/a$ not being equal to $0$ or $1/2$. Without loss of generality, we can 
assume that $t/a\in\{0,1/2\}$ or $t=s$. If $t/a=0$, then $\alpha$ is an element 
of $\mc{B}_k(S_1)$. If $t/a=1/2$ and $2 \nmid ab$, then we obtain a 
contradiction. If $t/a=1/2$ and $2\nmid k$, then both $a$ and $b$ are even, 
and we conclude that $\alpha$ is an element of $\mc{B}_k(S_1)$. It remains to 
consider the case $t=s$. In this case, we have
\[ \alpha=\varPhi\left(\frac{s}{a},\frac{s}{a}\right) = 
\varPhi\left(\frac{s}{a},\frac{qs}{a}\right) =
\varPhi\left(\frac{sc/a}{c},\frac{qsc/a}{c}\right)  = 
\varPhi\left(\frac{\hat{s}}{c},\frac{q\hat{s}}{c}\right) \]
for some integer $\hat{s}$. Thus, $\alpha$ is an element of $\mc{B}_k(S_4)$. A 
similar argument holds for the edge points of $\mc{B}_k(S_3)$, too. This proves 
the claim, and we conclude that 
\[ \text{Edge}\left( \bigcup \mc{B}_k(S_i) \right) =  \text{Edge}( 
\mc{B}_k(S_1) \cup \mc{B}_k(S_4) ) \]
unless $2|k$ and $2 | ab$. 

Now let us consider the exceptional case, namely $2|k$ and $2 | ab$. In this 
case, the integers $a$ and $b$ have different parity, i.e. one of them is odd 
and the other one is even. If $a\geq4$ is even then $\varPhi(1/a,1/2)$ is an 
edge point in $\mc{B}_k(S_2)$ that is not in $\mc{B}_k(S_1) \cup \mc{B}_k(S_4)$. 
Similarly, if $b\geq4$ is even then $\varPhi(1/b,1/2)$ is an edge point in 
$\mc{B}_k(S_3)$ that is not in $\mc{B}_k(S_1) \cup \mc{B}_k(S_4)$. The number 
of elements in this exceptional case is given by
\[ \varepsilon(k,q) = \left\{ 
\begin{array}{cl}
(b-b')(2-a')/2 + (a-a')(2-b')/2 & \textnormal{if }\ 
2\mathrel{|}k\ \textnormal{ and }\ 2\mathrel{|}ab, \\ 
0 & \textnormal{if }\ 2\nmid k \ \textnormal{ or }\ 2 \nmid ab.
\end{array}
\right.\]

The next step is to count the number of edge points. We start with picking an 
edge point from $\mc{B}_k(S_1)$. It may be of the form $\varPhi(0,m/b)$ or 
$\varPhi(m/a,0)$. If $a'=2$ or $b'=2$, then there are more edge points with one 
of the components being $1/2$. Thus, the number of edge points in 
$\mc{B}_k(S_1)$ is given by
\[\frac{(a-a')b' + (b-b')a'}{2}.\]

Now, we consider the edge points of $\mc{B}_k(S_4)$. Recall that an element in 
$\mc{B}_k(S_4)$ is of the form $\alpha = \varPhi(s/c,sq/c)$. The point $\alpha$ 
is an edge point if and only if $s \equiv \pm sq \pmod{c}$. This is true if and 
only if $s(q\pm1) \equiv 0 \pmod{c}$.
If $s$ is a multiple of $a$ or $b$, then this condition is satisfied. However, 
in some cases $c=2ab$, and therefore there are more pairs giving edge points. 
In 
total, the number pairs satisfying the conditions $s(q\pm1) \equiv 0 \pmod{c}$ 
is equal to
\[ c/a + c/b - 2\gcd(c/a,c/b) \]
Note that $\gcd(c/a,c/b)=c'$. The number of edge points in $\mc{B}_k(S_4)$ is 
found by dividing this number by two.

We finally prove that the intersection $\mc{B}_k(S_1) \cap \mc{B}_k(S_4)$ has 
only the corner points. To see this, it is enough to observe that 
$\varPhi(\sigma,\tau) = \varPhi(\tau,\sigma)$ is true for any  edge point in $ 
\mc{B}_k(S_4)$. However, this is not the case for the edge points in $ 
\mc{B}_k(S_1)$. In summary, the number of all edge points is
\[\frac{(a-a')b' + (b-b')a'}{2} + \frac{c/a + c/b - 2c'}{2} + \varepsilon(k,q). 
\]
Here the epsilon term is added in the exceptional case, namely $2|k$ and $2 | 
ab$. In any other case, its contribution is zero. 

There are three corner points and these corner points are always present if 
either $a'=2$ or $b'=2$. If otherwise, i.e. $a'b'=1$, then the number of corner 
points is $c'$.

We use the following implications in order establish the formulas for 
$\eta(k,q)$ in the theorem:
\begin{enumerate}
 \item $2\nmid k$ or $2\nmid c \Rightarrow c=ab$ and $a'=b'=c'=d'$.
  \item $2\mathrel{|}k$ and $2\mathrel{|}c$ and $2\mathrel{|}ab \Rightarrow 
c=2ab$ and $d'=1$ and $a'+b'=3$.
 \item $2\mathrel{|}k$ and $2\mathrel{|}c$ and $2\nmid ab \Rightarrow c=2ab$ and 
$a' = b' = d' =1$.
\end{enumerate}

Now, the quantity $\eta(k,q)$ is computed by adding up the five formulas in the 
table for the interior points with the number of edge points and the 
number of corner points.
\end{proof}

\section{The case G2}
Unless otherwise stated or proved, the assertions of this section can be found 
in \cite{bivariate}. For the convenience of the reader, we will summarize the 
main notions. Then we will prove the main result of this section, see 
Theorem~\ref{G2main}. 

Let $\{\alpha_1, \alpha_2\}$ be a choice of simple roots for the Lie algebra 
$G_2$ with Cartan matrix $\sm{2}{-1}{-3}{2}$.
The transpose of this matrix transforms the fundamental weights into the 
fundamental roots. We have $\alpha_1 = 2\omega_1-3\omega_2$ and $\alpha_2 = 
-\omega_1+2\omega_2$. The orbit of $\omega_1$, under the action of the Weyl 
group, is $\{ \pm \omega_1, \pm(\omega_1-\omega_2), \pm(2\omega_1-\omega) \}$. 
Similarly, the orbit of $\omega_2$ is $\{ \pm \omega_2, \pm(3\omega_1-\omega_2), 
\pm(3\omega_1-2\omega_2) \}$. We set $\sigma:=-\omega_1+\omega_2$ and 
$\tau=2\omega_1-\omega_2$. With this new choice, the orbits appear simpler. More 
precisely, we have $ \{ \pm \sigma, \pm \tau, \pm(\sigma+\tau)  \} $ and  $ \{ 
\pm(2\sigma+\tau), \pm(\sigma+2\tau), \pm (\sigma-\tau) \}$. One can consider 
$\varPhi=(\varphi_1,\varphi_2)$ with
\begin{align*}
 \varphi_1(\sigma,\tau) = &\ e^{2\pi i \sigma}+e^{2\pi i \tau}+e^{2\pi 
i(\sigma+\tau)}+e^{-2\pi i\sigma}+e^{-2\pi i \tau}+e^{-2\pi 
i(\sigma+\tau)}, \\
 \varphi_2(\sigma,\tau) = &\ e^{2\pi i (2\sigma+\tau)}+e^{2\pi i 
(\sigma+2\tau)}+e^{2\pi 
i(\sigma-\tau)} \\ &\ + e^{-2\pi i (2\sigma+\tau)}+e^{-2\pi i 
(\sigma+2\tau)}+e^{-2\pi 
i(\sigma-\tau)}. 
\end{align*}

Observe that $\varPhi(\sigma,\tau)$ is equal to any one of the following twelve 
expressions below which are given by the elements of the Weyl group: 
\[\begin{array}{ll|ll|ll}
\text{I} &\varPhi(\sigma,\tau) &\text{V} 
&\varPhi(\sigma,-\sigma-\tau) &\text{IX} 
& \varPhi(\tau,-\sigma-\tau)\\
\text{II} & \varPhi(\tau,\sigma) &\text{VI} 
& \varPhi(-\sigma-\tau,\sigma) 
&\text{X} & \varPhi(-\sigma-\tau,\tau)\\
\text{III} & \varPhi(-\sigma,-\tau) &\text{VII} 
& \varPhi(-\sigma,\sigma+\tau) & 
\text{XI} & \varPhi(-\tau,\sigma+\tau)\\
\text{IV} & \varPhi(-\tau,-\sigma) &\text{VIII} 
& \varPhi(\sigma+\tau,-\sigma) 
&\text{XII} & \varPhi(\sigma+\tau,-\tau)
\end{array}\]
Under these symmetries, the region $0 \leq \sigma, \tau < 1$ is separated into 
twelve triangles which are mutually congruent to each other under the action of 
the Weyl group. We choose one of them as $R_{G_2}$. See Figure~\ref{fig:RG2}.

\begin{figure}[htbp]
    \centering
 \includegraphics[scale=0.8]{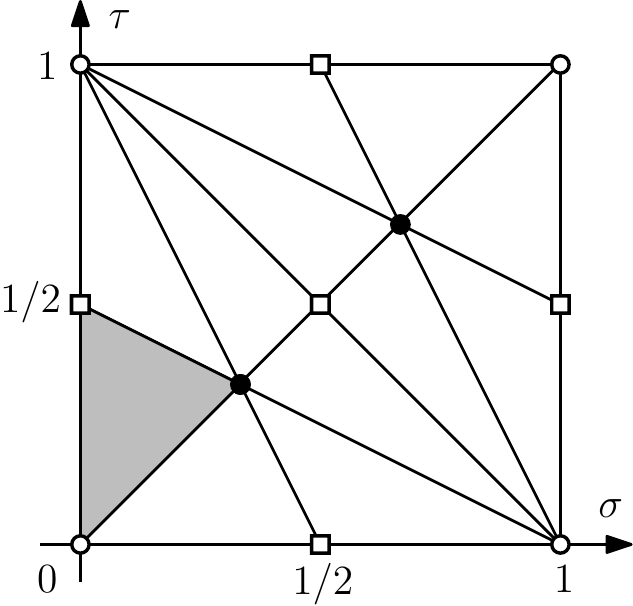}
    \caption{The fundamental region $R_{G_2}$.}
    \label{fig:RG2}
\end{figure}

The family of folding polynomials $\mc{G}_k$ satisfy the conditions 
\[\mc{G}_k(\varPhi(\sigma,\tau)) = \varPhi(k\sigma,k\tau).\]
We want to understand $\Fix(\mc{G}_q)$ in terms of $\varPhi(\cdot,\cdot)$. A 
fixed point of $\mc{G}_q$ is of the form $\varPhi(\sigma,\tau)$ where 
$(q\sigma,q\tau) \equiv w(\sigma,\tau) \pmod{\Z^2}$ for some $w$ in the Weyl 
group of $G_2$. Using this setup, it is not hard to show that $\Fix(\mc{G}_q)= 
\bigcup S_i$  where
\begin{align*}
 S_1 &= \left\{ \varPhi\left( \frac{s}{q-1}, \frac{t}{q-1} 
\right):s,t\in\Z \right\}, \\
 S_2 &= \left\{ \varPhi\left( \frac{s}{q^2-1}, \frac{sq}{q^2-1} 
\right):s\in\Z \right\}, \\
 S_3 &= \left\{ \varPhi\left( \frac{s}{q^2+q+1}, \frac{sq}{q^2+q+1} 
\right):s\in\Z \right\}, \\
 S_4 &= \left\{ \varPhi\left( \frac{s}{q+1}, \frac{t}{q+1} 
\right):s,t\in\Z \right\}, \\
 S_5 &= \left\{ \varPhi\left( \frac{s}{q^2-1}, \frac{s(-q)}{q^2-1} 
\right):s\in\Z \right\}, \\
 S_6 &= \left\{ \varPhi\left( \frac{s}{q^2-q+1}, \frac{s(-q)}{q^2-q+1} 
\right):s\in\Z \right\}.
\end{align*}

Let $K = \Q(\Fix(\mc{G}_q))$, a number field obtained by adjoining the 
solutions of $\mc{B}_q(x,y)=(x,y)$ to rational numbers. Let $\mf{p}$ be a prime 
ideal of $K$ lying over $p$. There is a one-to-one correspondence
\[ \F_q^2 \longleftrightarrow \Fix(\mc{G}_q) \]
obtained by reducing the elements on the right hand side modulo $\mf{p}$. 
Moreover this correspondence is compatible with the actions of $\bar{\mc{G}}_k$ 
and $\mc{G}_k$ on $\F_q^2$ and $\Fix(\mc{G}_q)$, respectively. Our purpose is 
to understand the size of
\[ \bar{\mc{G}}_k\left(\F_q^2\right)  = \bar{\mc{G}}_k\left( \overline{\bigcup 
S_i} \right) = \overline{ \bigcup \mc{G}_k(S_i) }.\]
Here $i$ runs from $1$ to $6$. In order to find the cardinality of the value 
set, it is enough to investigate the set of complex numbers $\bigcup 
\mc{G}_k(S_i)$.

\begin{theorem}\label{G2main}
Let $k$ be a positive integer. Set
\[a=\frac{q-1}{\gcd(q-1,k)},\ \ 
\tilde{a}=\frac{q+1}{\gcd(q+1,k)},\ \  
b=\frac{q^2-1}{\gcd(q^2-1,k)},\] 
\[c=\frac{q^2+q+1}{\gcd(q^2+q+1,k)}\quad \textnormal{and}\quad 
\tilde{c}=\frac{q^2-q+1}{\gcd(q^2-q+1,k)}.\]
Set $a'=\gcd(a,2)$ and $\tilde{a}'=\gcd(\tilde{a},2)$. Then the cardinality of 
the value set is 
\[\left|\bar{\mc{G}}_k\left(\F_q^2\right)\right|=\frac{a^2}{12}+\frac{\tilde{a}
^2}{12}+\frac{b}{2}+\frac{c }{6} + \frac{\tilde{c}}{6} +\eta(k ,q)\]
where $\eta(k,q)$ is given by
\[\begin{array}{|c|c|c|} \hline
\eta(k,q) & 3\nmid k \textnormal{ or } 3\nmid a\tilde{a} & 3\mathrel{|}k 
\textnormal{ and 
} 3\mathrel{|}a\tilde{a}\\ \hline
2 \nmid k \textnormal{ or } 2\nmid b & 0 & 1/3 \\ \hline
 2\mathrel{|}k \textnormal{ and } 2\mathrel{|}b & 
(a+\tilde{a})/2-1/(2a'\tilde{a}')  & 
(a+\tilde{a})/2 -1/(2a'\tilde{a}')+1/3 \\ \hline
 \end{array}\]
In particular if $\gcd(k,6)=1$, then $\eta(k,q)=0$.
\end{theorem}
\begin{proof}
Our strategy is to separate the problem into three parts according to 
the Figure~\ref{fig:RG2}. We will consider points in the interior, on the edge 
and at the corners separately.

A point $\varPhi(\sigma,\tau)$, with $0 \leq \sigma, \tau <1$, is a corner 
point if and only if it can expressed in the form $\varPhi(0,0), 
\varPhi(1/3,1/3)$ or $\varPhi(0,1/2)$. The first corner point $\varPhi(0,0)$ 
has a unique representation if we restrict $\sigma$ and $\tau$ to the region $0 
\leq \sigma,\tau <1$. However, this is not the case for the others. We have 
$\varPhi(1/3,1/3) = \varPhi(2/3,2/3)$ and $\varPhi(0,1/2) = \varPhi(1/2,0) = 
\varPhi(1/2,1/2)$.

A point $\varPhi(\sigma,\tau)$ is an edge point if it is given by a pair
$(\sigma,\tau)$ that is on the boundary of $R_{G_2}$ except the corners. An 
edge point can be expressed in the forms $\varPhi(\sigma,\sigma)$ or 
$\varPhi(0,\sigma)$. There are six different expressions for each edge point 
if we restrict $(\sigma, \tau)$ to the region $0 \leq \sigma,\tau <1$. 

A point $\varPhi(\sigma,\tau)$ is an interior point if it is given by a pair
$(\sigma,\tau)$ that is in the interior of $R_{G_2}$. There are exactly twelve 
distinct representations given by I, II, ..., XII, when the components 
are restricted modulo integers. 

To ease the notation, we set $m'=\gcd(m,2)$ and $m''=\gcd(m,3)$ for any integer 
$m$. We have the following table:
\[\begin{array}{|l|c|} \hline
× & \text{Interior}\\ \hline
\mc{G}_k(S_1) & (a^2-6a+3a'+2a'')/12 \\ \hline
\mc{G}_k(S_2) & (b-b/a-b/\tilde{a}+(b/a,b/\tilde{a}))/4  \\ \hline
\mc{G}_k(S_3) & (c-c'')/6 \\ \hline
\mc{G}_k(S_4) & (\tilde{a}^2-6\tilde{a} + 3\tilde{a}' + 2\tilde{a}'')/12\\ 
\hline
\mc{G}_k(S_5) & (b-b/a-b/\tilde{a}+(b/a,b/\tilde{a}))/4\\ \hline
\mc{G}_k(S_6) & (\tilde{c}-\tilde{c}'')/6 \\ \hline
\end{array}\]

The elements in $\mc{G}_k(S_1)$ are of the form $\varPhi(s/a,t/a)$ for 
some integers $s$ and $t$. There are $a^2$ pairs $(s/a,t/a)$ with $0\leq s,t < 
a$ and as a result there are roughly $a^2/12$ interior points in 
$\mc{A}_k(S_1)$. In order to find the precise number, we need to exclude the 
pairs giving edge and corner points. The number of pairs, which give interior 
points, is equal to
\[ a^2-6(a-a'-a''+1)-3(a'-1)-4(a''-1)-1 = a^2-6a+3a'+2a''.\]
In order to see this, we use the following idea: If $s/a=1/2,1/3$ or $2/3$, 
then we consider those cases separately. This explains the second 
term on the left. If $a'=2$, then we have to exclude only three pairs, 
namely $(1/2,0),(0,1/2)$ and $(1/2,1/2)$. This explains the third term, namely 
$3(a'-1)$. If $a''=3$, then we have to exclude eight pairs, namely $(\pm 
1/3,\pm 1/3),(0,\pm1/3)$ and $(\pm1/3,0)$. This explains the fourth term. The 
fifth and the last term $-1$ is for excluding the pair $(0,0)$.

Secondly, we consider the interior points in $\mc{G}_k(S_2)$. The elements in 
$\mc{G}_k(S_2)$ are of the form $\varPhi(s/b,sq/b)$ for some integer $s$. There 
are $b$ such pairs with $0\leq s < b$. An interior point of this form must have 
four distinct representations:
\[\varPhi\left(\frac{s}{b}, \frac{sq}{b}\right) = \varPhi\left(\frac{sq}{b}, 
\frac{s}{b}\right)= \varPhi\left(\frac{-s}{b}, \frac{-sq}{b}\right)= 
\varPhi\left(\frac{-sq}{b}, \frac{-s}{b}\right).\]
Here the numerators are considered modulo $b$. Note that $\varPhi(a/b,qa/b)$ 
is equal to $\varPhi(a/b,0)$ and it is an edge point. To see this, note that 
\[a+qa \equiv a(q+1) \equiv 0 \pmod{b}.\]
There are $b/a$ distinct pairs $(ma/b,mqa/b)$ modulo integers for $m=1,\ldots, 
b/a$ which have the same property. Similarly, there are $b/\tilde{a}$ distinct 
pairs $(m\tilde{a}/b,mq\tilde{a}/b)$ modulo integers for $m=1,\ldots, b/a$ 
which give edge points. Applying the inclusion and exclusion principle, we 
justify the second row of the table.

The third row, namely the number of interior points in $\mc{G}_k(S_3)$, is 
relatively easier to compute. The elements in $\mc{G}_k(S_3)$ are of the form 
$\varPhi(s/c,sq/c)$ for some integer $s$. There are six different pairs 
giving the same point, namely
\[\pm\left(\frac{s}{c}, \frac{sq}{c}\right),\ \pm\left(\frac{sq}{c}, 
\frac{sq^2}{c}\right),\ \pm\left(\frac{sq^2}{c}, \frac{s}{c}\right).\]
Here the numerators are considered modulo $c$. The corner point
$\varPhi(1/3,1/3)$ can be expressed in the form $\varPhi(s/c,sq/c)$ if and only 
if $c''=3$. Thus, there are $c-c''$ pairs which give an interior point. 

The other half of the table is obtained in a similar fashion. We also note 
that the set of interior points of $\mc{G}_k(S_i)$ are pairwise disjoint. We 
will explain this for one pair and omit the others. Let us pick a point
$ \alpha \in \mc{G}_k(S_2) \cap \mc{G}_k(S_5)$. 
We will show that $\alpha$ is either an edge point or a corner point. We have
\[\alpha = \varPhi\left(\frac{s}{b}, \frac{sq}{b}\right) = 
\varPhi\left(\frac{t}{b}, \frac{t(-q)}{b}\right)\]
for some integers $s$ and $t$. Without loss of generality, we can assume that 
$s\equiv t \pmod{ b}$. It follows that $sq$ is either congruent to $s(-q)$, or 
congruent to $-s+sq$ modulo $b$. This is obtained by either I or V, 
respectively. In either case we have $s=0$, and therefore $\alpha$ is the corner 
point $\varPhi(0,0)$. This finishes the discussion for the interior points.

Now, we focus on the edge points. We first note that $\mc{G}_k(S_3)$ and 
$\mc{G}_k(S_6)$ have no such points. Pick an edge point $\varPhi(s/a,t/a)$ 
from $\mc{G}_k(S_1)$. Without loss of generality, we can assume that $t=s$ or 
$t=0$. In the former case, we have $\varPhi(s/a,s/a)=\varPhi(s/a,qs/a)$, an 
element of $\mc{G}_k(S_2)$. In the latter case, we have 
$\varPhi(s/a,0)=\varPhi(s/a,-qs/a)$, an 
element of $\mc{G}_k(S_5)$. A similar argument holds for the edge 
points of $\mc{G}_k(S_4)$, too. Thus, we conclude that
\[ \text{Edge}( \mc{G}_k(S_1) \cup \mc{G}_k(S_4) ) \subseteq  \text{Edge}( 
\mc{G}_k(S_2) \cup \mc{G}_k(S_5) ), \]
and therefore
\[ \text{Edge}\left( \bigcup \mc{G}_k(S_i) \right) =  \text{Edge}( 
\mc{G}_k(S_2) \cup \mc{G}_k(S_5) ). \]
Pick an edge point from $\mc{G}_k(S_2)$. It is of the form either 
$\varPhi(ma/b,mqa/b)$ with $m=1,\ldots,b/a$, or $\varPhi(m\tilde{a}/b, 
mq\tilde{a}/b)$ with $m=1,\ldots,b/\tilde{a}$. If $ma/b$  or 
$m\tilde{a}/b$ is equal to one of $1/2,1/3$ or $2/3$, then we consider those 
cases separately. In this separate case, the only possibility is the edge point 
$\varPhi(1/3,2/3)$. Thus the number of edge points in $\mc{G}_k(S_2)$, other 
than $\varPhi(1/3,2/3)$, is precisely  
\[ \frac{b/a-(b/a)' - (b/a)''+ 1 + b/\tilde{a} - (b/\tilde{a})' - 
(b/\tilde{a})''+1 }{2} \]
The same value holds for $\mc{G}_k(S_5)$, too. Note that $\varPhi(1/3,2/3)$ is 
present if and only if $a''+\tilde{a}''=4$. Thus the number of all edge points 
is equal to
\begin{equation}\label{g2edge}
 b/a-(b/a)' - (b/a)''+ 1 + b/\tilde{a} - (b/\tilde{a})' - 
(b/\tilde{a})''+1 + \frac{a''+\tilde{a}''-2}{2}.
\end{equation}
This finishes the discussion for the edge points.

There are three corner points, namely $\varPhi(0,0), \varPhi(0,1/2)$ and 
$\varPhi(1/3,1/3)$. The number of corner points that are present in $\bigcup 
\mc{G}_k(S_i)$ is
\begin{equation}\label{g2corner}
1 + ((a\tilde{a}b)'-1)+(a''+\tilde{a}''-2)/2.
\end{equation}
In this sum with three terms, the middle term is equal to one if and only if 
$\varPhi(0,1/2)$ is present, otherwise it is zero. Similarly 
$(a''+\tilde{a}''-2)/2$ is equal to one if and only if $\varPhi(1/3,1/3)$ is 
present, otherwise it is zero. This finishes the discussion for the corner 
points.

We use the following implications in order establish the formulas for 
$\eta(k,q)$ in the theorem:
\begin{enumerate}
 \item $2 \nmid k$ or $2\nmid b\Rightarrow b=a\tilde{a}$ and 
$a'=\tilde{a}'=(b/a)'=(b/\tilde{a})'=\gcd(b/a,b/\tilde{a})$
 \item $2 \mathrel{|} k$ and $2 \mathrel{|} b\Rightarrow 
b=2a\tilde{a}$ and $(b/a)'=(b/\tilde{a})'=2$ and 
$\gcd(b/a,b/\tilde{a})=2$.
 \item $3 \nmid k$ or $3\nmid a\tilde{a}\Rightarrow a''=c''$ and 
$\tilde{a}''=\tilde{c}''$
 \item $3 \mathrel{|} k$ and $3 \mathrel{|} a\tilde{a}\Rightarrow 
c''=\tilde{c}''=1$ and $(a\tilde{a})''+1=a''+\tilde{a}''=4$.
\end{enumerate}

Now, the quantity $\eta(k,q)$ is computed by adding up the six formulas in the 
table for the interior points with the expression for the edge 
points, see the equation~(\ref{g2edge}), and the expression for the corner 
points, see the equation~(\ref{g2corner}). There is a subtle point if (2) 
holds; if the sum $a'+\tilde{a}'$ is equal to three, then $\eta(k,q) = 
(a+\tilde{a})/2-1/4$, and if the sum $a'+\tilde{a}'$ is two, then $\eta(k,q) = 
(a+\tilde{a})/2-1/2$. In order to keep table for $\eta(k,q)$ as simple as 
possible, we write these two different expressions in one single formula. More 
precisely, we write $\eta(k,q) = (a+\tilde{a})/2-1/(2a'\tilde{a}')$ which covers 
both cases.
\end{proof}

{\small
\def\refname{References}
\newcommand{\etalchar}[1]{$^{#1}$}

\end{document}